\documentclass[letterpaper]{article}
\usepackage{uai20202}
\usepackage[margin=1in]{geometry}

\usepackage{multicol}
\usepackage{multirow}
\usepackage{tabularx}

\usepackage{times}
\usepackage{algpseudocode}% http://ctan.org/pkg/algorithmicx
\usepackage{lipsum}% http://ctan.org/pkg/lipsum
\usepackage{float}% http://ctan.org/pkg/float
\usepackage{amsmath, amsthm, amssymb , bbm}
\newtheorem{theorem}{Theorem}[section]
\newtheorem{corollary}{Corollary}[theorem]
\newtheorem{lemma}[theorem]{Lemma}
\newtheorem{proposition}[theorem]{Proposition}
\theoremstyle{definition}
\newtheorem{definition}{Definition}[section]
\newtheorem{remark}{Remark}
\newtheorem*{lemma35}{Lemma 3.5}
\newtheorem*{lemma37}{Lemma 3.7}

\newtheorem*{theorem38}{Theorem 3.8}
\newtheorem*{corollary381}{Corollary 3.8.1}

\newtheorem*{proposition41}{Proposition 4.1}
\newtheorem*{proposition42}{Proposition 4.2}

\usepackage[toc,page]{appendix}
\usepackage[round]{natbib}
\usepackage{mathtools, nccmath}
\DeclareMathOperator{\head}{head}
\DeclareMathOperator{\tail}{tail}
\DeclareMathOperator{\pa}{pa}
\DeclareMathOperator{\sib}{sib}
\DeclareMathOperator{\an}{an}
\DeclareMathOperator{\de}{de}
\DeclareMathOperator{\dis}{dis}
\DeclareMathOperator{\barren}{barren}
\DeclareMathOperator{\ant}{ant}

\DeclarePairedDelimiter{\abs}{\lvert}{\rvert}
\newcommand{\lqarrow}{\mathbin{\leftarrow\!\!\medmath{?}}}
\newcommand{\rqarrow}{\mathbin{\medmath{?}\!\!\rightarrow}}
\usepackage[linesnumbered]{algorithm2e}

\newcommand\indep{\protect\mathpalette{\protect\independenT}{\perp}}
\def\independenT#1#2{\mathrel{\rlap{$#1#2$}\mkern2mu{#1#2}}}

\usepackage{tikz}
\usetikzlibrary{shapes, arrows}
\usepackage{graphicx}
\usepackage{multirow}
\title{Faster algorithms for Markov equivalence}

\author{ {\bf Zhongyi Hu} \\
Department of Statistics \\
University of Oxford\\
zhongyi.hu@keble.ox.ac.uk\\
\And
{\bf Robin Evans}  \\
Department of Statistics          \\
University of Oxford \\
evans@stats.ox.ac.uk
}
\begin{document}
\maketitle

\begin{abstract}%   <- trailing '%' for backward compatibility of .sty file
   Maximal ancestral graphs (MAGs) have many desirable properties; in particular they can fully describe conditional independences from directed acyclic graphs (DAGs) in the presence of latent and selection variables. However, different MAGs may encode the same conditional independences, and are said to be \emph{Markov equivalent}. Thus identifying necessary and sufficient conditions for equivalence is essential for structure learning. Several criteria for this already exist, but in this paper we give a new non-parametric characterization in terms of the heads and tails that arise in the parameterization for discrete models. We also provide a polynomial time algorithm ($O(ne^{2})$, where $n$ and $e$ are the number of vertices and edges respectively) to verify equivalence.  Moreover, we extend our criterion to ADMGs and summary graphs and propose an algorithm that converts an ADMG or summary graph to an equivalent MAG in polynomial time ($O(n^{2}e)$).  Hence by combining both algorithms, we can also verify equivalence between two summary graphs or ADMGs.
\end{abstract}

% \begin{keywords}
%  acyclic directed mixed graph, Bayesian networks, Markov equivalence, maximal ancestral graphs, summary graphs
% \end{keywords}

\section{INTRODUCTION}
DAG models, also known as Bayesian networks, are popular graphical models that associate a probability distribution $P(X_{V})$ with a graph consisting of vertices representing random variables $X_{V}$ joined by  directed edges. In the context of causal inference, a directed edge $a \rightarrow b$ can be interpreted as `$a$ has a direct causal effect on $b$'. A DAG encodes conditional independence in $P$ by a criterion called d-separation \citep{pearl2009causality}. For example, $1 \rightarrow 2 \leftarrow 3$ is a DAG with vertices $1,2,3$ and implies one independence: $X_1 \indep X_3$.  DAGs are also associated with an elegant factorization of probability distributions, which allows fast statistical inference and fitting. With some additional assumptions they can be used for causal modelling, and thus they are used in many fields such as expert systems, pattern recognition in machine learning, or estimating causal effects in experimental science. 
\par
An interesting question is how to learn unknown DAGs from a dataset. \citet{spirtes2000causation} provide an algorithm called the PC algorithm; this learns the underlying DAG by testing conditional independences inherited in the data. However, when latent variables are present, conditional independence in the observed variables may imply the wrong underlying causal structure, or even not correspond to any DAGs at all. For example, in Figure \ref{LDAG}(i) with latent variable $h$ (this is an example from \citet{richardson2002}), there is no DAG that describes precisely the independence on the marginal. We say, then, that DAGs are not closed under  marginalization. Classes of supermodels have been developed to tackle this problem, one of which is \emph{maximal ancestral graphs} (MAGs) introduced by \citet{richardson2002}.  This includes graphs with additional types of edges: bidirected edges ($\leftrightarrow$) and undirected  edges. A bidirected edge $1 \leftrightarrow 2$ can be interpreted as saying that there is a latent variable $h$ such that $1 \leftarrow h \rightarrow 2$. An undirected edge arises when there are some variables being conditioned upon. Graphical implications for conditional independence are extended from d-separation to m-separation, see definitions in Section \ref{def}. Moreover one can project a DAG with latent and selection variables to a Markov equivalent MAG on the observed margin. The projection is described in Section \ref{OMESG}. The resulting MAG not only captures the exact conditional independence of observed variables in the original graph but also preserves ancestral relations. In addition, Gaussian variables associated with MAGs are curved exponential families \citep{richardson2002}, and hence have some desirable statistical properties. 
\par
Graphs in this paper are directed and contain no undirected edge. Extensions to \emph{summary graphs} and MAGs with undirected edges are straightforward and we have placed them in the supplementary materials. Note that summary graphs defined in \citet{wermuth2011} are actually the same as ADMGs with undirected components at the top. Graphically, one just needs to change the dashed lines to bidirected edges and they encode the same conditional independence. We include details of this discussion in the supplementary materials.
\par
Learning causal structures via testing only conditional independence leads to another problem. Different MAGs can imply the same set of constraints on variables, for example Figures \ref{MEF}(i) and (ii) both only encode $X_1 \indep X_4 \mid X_2,X_3$. We say such MAGs are \emph{Markov equivalent}, and are in the same \emph{Markov equivalence class}. Thus equivalent graphs represent the same set of distributions. Although each class can be uniquely described by a \emph{partial ancestral graph} (PAG) \citep{colombo2012learning}, non-experimental data cannot distinguish graphs in the same class.  Therefore identifying conditions for Markov equivalence is important for modelling and estimating causal effects from data. 
\par
There have been three graphical characterizations that give necessary and sufficient conditions for when two MAGs are equivalent. Among those three criteria, only \citet{ali2009} provide a polynomial time algorithm to verify Markov equivalence. \citet{Zhao2005} characterize MAGs by \emph{minimal collider paths} (MCPs).  The criterion of \citet{Spirtes97apolynomial} uses \emph{discriminating paths}, which we will define in Section 3 (we will employ them in our proofs). This paper gives a new characterization and it lead to a faster algorithm to test equivalence compared to existing ones. Also we show a similar equivalence criterion for wider classes of acyclic graphs, ADMGs. 
 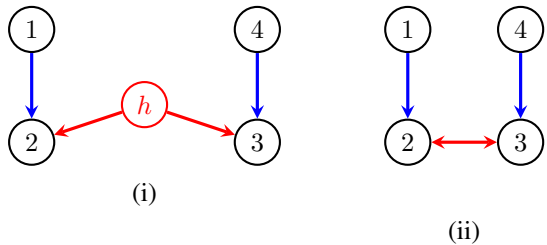
\begin{figure}[ht]
   \begin{tikzpicture}
   [rv/.style={circle, draw, thick, minimum size=6mm, inner sep=0.8mm}, node distance=15mm, >=stealth]
   \pgfsetarrows{latex-latex};
 \begin{scope}
   \node[rv]  (1)            {$1$};
   \node[rv, below of=1] (2) {$2$};
   \node[rv, right of=2, yshift=5mm, red] (3) {$h$};
   \node[rv, right of=3, yshift=-5mm] (4) {$3$};
   \node[rv, above of=4] (5) {$4$};
   \draw[->, very thick, blue] (1) -- (2);
   \draw[<-, very thick, red] (2) -- (3);
   \draw[->, very thick, red] (3) -- (4);
   \draw[<-, very thick, blue] (4) -- (5);
   \node[below of=3, yshift=0.3cm] {(i)};
   \end{scope}
 \begin{scope}[xshift = 5cm]
   \node[rv]  (1)                        {$1$};
   \node[rv, below of=1] (2)             {$2$};
   \node[rv, right of=2, xshift=0cm] (3) {$3$};
   \node[rv, above of=3] (4)             {$4$};
   \draw[->, very thick, blue] (1) -- (2);
   \draw[<->, very thick, red] (2) -- (3);
   \draw[<-, very thick, blue] (3) -- (4);
   \node[below of=3, xshift=-7.5mm, yshift=0.3cm] {(ii)};
   \end{scope}
 \end{tikzpicture}
  \caption{(i) A DAG with latent variable $h$. (ii) A Markov equivalent MAG of (i) on the margin $\{1,2,3,4\}$.}
 \label{LDAG}
 \end{figure}
 \\
In Section \ref{def}, we give basic definitions and terminology for graphical models. In Section \ref{sec:mainres} we present the main results, including theorems on the Markov equivalence of MAGs and ADMGs. Algorithms to verify Markov equivalence and their complexities are shown in Section \ref{algorithm}.  Missing proofs are found in the appendix.

\section{DEFINITIONS}\label{def}
\subsection{Graphs}
A \emph{graph} $\mathcal{G}$ consists of a vertex set $\mathcal{V}$ and an edge set $\mathcal{E}$ of distinct pairs of vertices. For an edge in $\mathcal{E}$ connecting vertices $a$ and $b$, we say these two vertices are the \emph{endpoints} of the edge and the two vertices are \emph{adjacent} (if there is no edge between $a$ and $b$, they are \emph{nonadjacent}).
\par

A \emph{path} is a set of distinct vertices $v_{i}, 1 \leq i \leq k$ such that $v_{i}$ and $v_{i+1}$ is connected by some edge for all $i \leq k-1$. A path is \emph{directed} if its edges are all directed and point in the same direction. A graph $\mathcal{G}$ is \emph{acyclic} if there is no directed cycle (any \emph{directed path} such that $v_{1} \rightarrow v_{2} \rightarrow ... \rightarrow v_{k}$ and $v_{k} \rightarrow v_{1}$). A \emph{graph} $\mathcal{G}$ is called an \emph{acyclic directed mixed graph} (ADMG) if it is \emph{acyclic} and contains only directed and bidirected edges. 
\par
For a vertex $v$ in an ADMG $\mathcal{G}$, we define the following sets:
\begin{align*}
\pa_{\mathcal{G}}(v) &= \{w: w \rightarrow v \text{ in } \mathcal{G}\}\\
\sib_{\mathcal{G}}(v) &= \{w:w \leftrightarrow v \text{ in } \mathcal{G}\}\\
\an_{\mathcal{G}}(v) &= \{w: w \rightarrow ... \rightarrow v \text{ in } \mathcal{G} \text{ or } w=v\}\\
\de_{\mathcal{G}}(v) &= \{w: v \rightarrow ... \rightarrow w \text{ in } \mathcal{G} \text{ or } w=v\}\\
\dis_{\mathcal{G}}(v) &= \{w: w \leftrightarrow ... \leftrightarrow v \text{ in } \mathcal{G} \text{ or } w=v\}.
\end{align*}
They are known as the \emph{parents}, \emph{siblings}, \emph{ancestors}, \emph{descendants} and \emph{district} of $v$, respectively. These sets are also defined disjunctively for a set of vertices $W \subseteq \mathcal{V}$. For example $\pa_{\mathcal{G}}(W) = \bigcup_{w \in W} \pa_{\mathcal{G}}(w)$. Vertices in the same district are connected by a bidirected path and this is an equivalence relation, so we can partition $\mathcal{V}$ and denote the \emph{districts} of a \emph{graph} $\mathcal{G}$ by $\mathcal{D}(\mathcal{G})$. We sometimes ignore the subscript if the graph we refer to is clear, for example $\an (v)$ instead of $\an_{\mathcal{G}}(v)$.

\subsection{Separation Criterion}
For a path $\pi$ with vertices $v_{i}$, $1 \leq i \leq k$ we call $v_{1}$ and $v_{k}$ the \emph{endpoints} of $\pi$ and any other vertices the \emph{nonendpoints} of $\pi$. For a nonendpoint $w$ in $\pi$, it is a \emph{collider} if $\rqarrow$ $w$ $\lqarrow$ on $\pi$ and a \emph{noncollider} otherwise (an edge $\rqarrow$ is either $\rightarrow$ or $\leftrightarrow$). For two vertices $a,b$ and a disjoint set of vertices $C$ in $\mathcal{G}$ ($C$ might be empty), a path $\pi$ is \emph{m-connecting} $a,b$ given $C$ if (i) $a,b$ are endpoints of $\pi$, (ii) every noncollider is not in $C$ and (iii) every collider is in $\an_{\mathcal{G}}(C)$. A \emph{collider path} is a path where all the nonendpoints vertices are colliders.

\begin{definition}
For three disjoint sets $A,B$ and set $C$ ($A,B$ are non-empty), $A$ and $B$ are \emph{m-separated} by $C$ in $\mathcal{G}$ if there is no m-connecting path between any $a \in A$ and any $b \in B$ given $C$. We denote the m-separation by $A \perp_m B \mid C$. 
\end{definition}

For a triple ($a,b,c$) in a graph $\mathcal{G}$, we call this an \emph{unshielded triple} if $\{a,b\}$ and $\{b,c\}$ are adjacent but $\{a,c\}$ are not. If $b$ is a also collider in the path $\langle a,b,c \rangle$ then we also call the triple an \emph{unshielded collider}, and an \emph{unshielded noncollider} otherwise.

\begin{definition}\label{ordinary}
A distribution $P(X_{V})$ is said to be in the \emph{Markov model} of an ADMG $\mathcal{G}$ if whenever $A \perp_m B\mid C$ in $\mathcal{G}$, $X_{A} \indep X_{B}\mid X_{C}$ in $P$.
\end{definition}
This definition, known as the global Markov property, associates distributions with a given ADMG via m-separations. There are also other equivalent definitions in terms of the local Markov property or moralization  \citep[see][]{richardlocalmarkov}, but the global Markov property has the advantage of being 
`complete': that is, if there is no m-separation then almost every distribution
in the model does not satisfy the associated conditional independence.
\begin{remark}
The model in Definition \ref{ordinary} defined by conditional independences is sometimes referred as the \emph{ordinary Markov model}. There is a model called the \emph{nested Markov model} defined by generalized conditional independences which captures all the equality constraints that arise from latent variable model \citep[see][]{richardson2017nested, Evans_2018}.
\end{remark}
For an ADMG $\mathcal{G}$, given a subset $W \subseteq \mathcal{V}$, the induced subgraph $\mathcal{G}_{W}$ is defined as the graph with vertex set $W$ and edges in $\mathcal{G}$ whose endpoints are both in $W$. Also for the \emph{district} of a vertex $v$ in an induced subgraph $\mathcal{G}_{W}$, we may denote it by $\dis_{W}(v)$. 
\subsection{MAGs}
\begin{definition}\label{maximal}
An ADMG $\mathcal{G}$ is \emph{maximal} if for every pair of \emph{nonadjacent} vertices $a$ and $b$, there exists some set $C$ such that $a,b$ are m-separated given $C$ in $\mathcal{G}$.
\end{definition}

\begin{definition}
An ADMG $\mathcal{G}$ is \emph{ancestral} if for every $v \in \mathcal{V}$, $\sib_{\mathcal{G}}(v) \cap \an_{\mathcal{G}}(v) = \emptyset$.
\end{definition}

\begin{definition}
An ADMG $\mathcal{G}$ is called a \emph{maximal ancestral graph} (MAG) if it is \emph{maximal} and \emph{ancestral}.
\end{definition}
Note that in an ancestral graph, there is at most one edge between each pair of vertices.

\begin{figure}
  \begin{tikzpicture}
  [rv/.style={circle, draw, thick, minimum size=6mm, inner sep=0.8mm}, node distance=14mm, >=stealth]
  \pgfsetarrows{latex-latex};
\begin{scope}
  \node[rv]  (1)            {$1$};
  \node[rv, right of=1] (2) {$2$};
  \node[rv, below of=1] (3) {$3$};
  \node[rv, right of=3] (4) {$4$};
  \draw[<->, very thick, red] (1) -- (3);
  \draw[<->, very thick, red] (2) -- (4);
  \draw[<->, very thick, red] (3) -- (4);
  \draw[->, very thick, color=blue] (3) -- (2);
  \draw[->,very thick, blue] (4) -- (1);
  \node[below right of=3,xshift=-0.3cm] {(i)};
  \end{scope}
\begin{scope}[xshift = 3cm]
   \node[rv]  (1)           {$1$};
  \node[rv, right of=1] (2) {$2$};
  \node[rv, below of=1] (3) {$3$};
  \draw[<->, very thick, red] (1) -- (3);
  \draw[->, very thick, blue] (1) -- (2);
  \draw[->, very thick, blue] (2) -- (3);
  \node[below right of=3,xshift=-0.3cm] {(ii)};
  \end{scope}
 \begin{scope}[xshift = 6cm]
   \node[rv]  (1)           {$1$};
  \node[rv, right of=1] (2) {$2$};
  \node[rv, below of=1] (3) {$3$};
  \node[rv, right of=3] (4) {$4$};
  \draw[<->, very thick, red] (2) -- (4);
  \draw[<->, very thick, red] (3) -- (4);
  \draw[->, very thick, color=blue] (3) -- (2);
  \draw[->,very thick, blue] (4) -- (1);
  \draw[<->, very thick,red] (1) -- (2);
  \node[below right of=3,xshift=-0.3cm] {(iii)};
  \end{scope}
\end{tikzpicture}
\caption{(i) An ancestral graph that is not maximal. (ii) A maximal graph that is not ancestral. (iii) A maximal ancestral graph. }
\label{MAGs}
\end{figure}
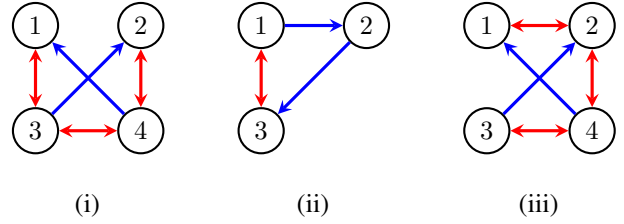

For example, the graph in Figure \ref{MAGs}(i) is not maximal because 1 and 2 are not adjacent, but no subset of $\{3,4\}$ will m-separate them. (ii) is not ancestral as 1 is a sibling of 3, which is also one of its descendants. (iii) is a MAG in which the only conditional independence is $X_1 \indep X_3 \mid X_4$.
\par
\begin{definition}
Two graphs $\mathcal{G}_{1}$ and $\mathcal{G}_{2}$ with the same vertex sets, are said to be \emph{Markov equivalent} if any m-separation holds in $\mathcal{G}_{1}$ if and only if it holds in $\mathcal{G}_{2}$.
\end{definition}

\subsection{Heads and Tails}
For a vertex set $W \subseteq \mathcal{V}$, we define the \emph{barren subset} of $W$ as:$$\barren_{\mathcal{G}}(W) = \{w \in W:\de_{\mathcal{G}}(w) \cap W = \{w\}\}.$$
A vertex set $H$ is called a \emph{head} if (i) $\barren_{\mathcal{G}}(H) = H$ and (ii) $H$ is contained in a single district in $\mathcal{G}_{\an (H)}$. For an ADMG $\mathcal{G}$, we denote the set of all heads in $\mathcal{G}$ by $\mathcal{H}(\mathcal{G})$. A \emph{tail} of a $\head$ is defined as:$$\tail(H) = (\dis_{\an(H)}(H) \setminus H) \cup \pa_{\mathcal{G}}(\dis_{\an(H)}(H)).$$
Distributions associated with an Markov model can be factorized in terms of heads and tails \citep{DBLP:journals/corr/Richardson14}.

\begin{definition}
The \emph{parametrizing set} of $\mathcal{G}$, denoted by $\mathcal{S}(\mathcal{G})$ is defined as:
$$\mathcal{S}(\mathcal{G}) = \{H\cup A:H \in \mathcal{H}(\mathcal{G})\text{ and } \emptyset \subseteq A \subseteq \tail(H)\}.$$
\end{definition}

Note that it is called the parametrizing set because it is closely related to the discrete parameterization \citep{Evans2014}. However the theorem developed in this paper is entirely non-parametric. We also define $\mathcal{S}_{k}(\mathcal{G})$ for $k \geq 2$ as: 
$$\mathcal{S}_{k}(\mathcal{G}) = \{S \in \mathcal{S}(\mathcal{G}): 2 \leq \abs{S} \leq k\}.$$ 
In particular, we are interested in: 
\begin{align*}
\Tilde{\mathcal{S}_{3}}(\mathcal{G}) &= \{S \in \mathcal{S}_{3}(\mathcal{G}) \mid \text{there are 1 or 2 adjacencies} \\ 
&\qquad \qquad \text{among the vertices in }S\}.
\end{align*}
We write $\mathcal{S}, \mathcal{S}_{k}, \Tilde{\mathcal{S}}_{3}$ if the graph $\mathcal{G}$ we are referring to is clear. Note that we are not considering any singleton sets in $\mathcal{S}_k(\mathcal{G})$ or $\Tilde{\mathcal{S}}_3(\mathcal{G})$; these are just all vertices because $\{v\}$ is trivially a head. For a MAG $\mathcal{G}$, a pair of vertices are in $\mathcal{S}(\mathcal{G})$ if and only if they are adjacent (This is easy to prove).
\par
We give an example to illustrate what the sets defined above are. Consider the three MAGs in Figure \ref{MEF}, Table \ref{tab1} lists their heads and tails, Table \ref{tab2} lists their parametrizing sets $\mathcal{S}$ and Table \ref{tab3} lists their $\mathcal{S}_{3}$ and $\Tilde{\mathcal{S}}_{3}$.
% Moreover, let $\mathcal{G}_{1}$, $\mathcal{G}_{2}$ and $\mathcal{G}_{3}$ denote Figure \ref{MEF} (i), (ii) and (iii) respectively. Then it is easy to see that $\mathcal{S}_{3}(\mathcal{G}_{1})$ excludes the singleton sets in $\mathcal{S}(\mathcal{G}_{1})$ (so does $\mathcal{S}_{3}(\mathcal{G}_{2}))$ 
% and $\mathcal{S}_{3}(\mathcal{G}_{3}))$) while $\mathcal{S}_{3}(\mathcal{G}_{3})$ also excludes $\{1,2,3,4\}$ from $\mathcal{S}(\mathcal{G}_{3})$. Further $\Tilde{\mathcal{S}}_{3}(\mathcal{G}_{1}) = \Tilde{\mathcal{S}}_{3}(\mathcal{G}_{2})$ only keeps sets of size two since $\{1,2,3\}$ and $\{2,3,4\}$ are triples with full adjacencies. On the other hand, apart from those sets of size two, $\Tilde{\mathcal{S}}_{3}(\mathcal{G}_{3})$ keeps $\{1,3,4\}$ as $1,4$ are not adjacent.
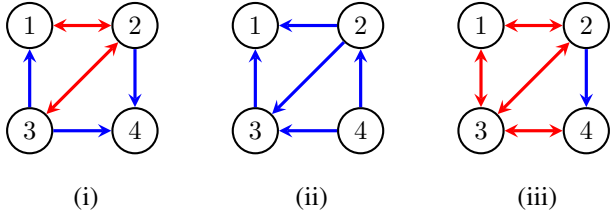
\begin{figure}
  \begin{tikzpicture}
  [rv/.style={circle, draw, thick, minimum size=6mm, inner sep=0.8mm}, node distance=14mm, >=stealth]
  \pgfsetarrows{latex-latex};
\begin{scope}
  \node[rv]  (1)            {$1$};
  \node[rv, right of=1] (2) {$2$};
  \node[rv, below of=1] (3) {$3$};
  \node[rv, right of=3] (4) {$4$};
  \draw[<->, very thick, red] (1) -- (2);
  \draw[<->, very thick, red] (2) -- (3);
  \draw[->, very thick, color=blue] (3) -- (4);
  \draw[->,very thick, blue] (3) -- (1);
  \draw[->,very thick, blue] (2) -- (4);
  \node[below of=3, xshift=7.5mm, yshift=0.5cm] {(i)};
  \end{scope}
\begin{scope}[xshift = 3cm]
  \node[rv]  (1)            {$1$};
  \node[rv, right of=1] (2) {$2$};
  \node[rv, below of=1] (3) {$3$};
  \node[rv, right of=3] (4) {$4$};
  \draw[<-, very thick, blue] (1) -- (2);
  \draw[->, very thick, blue] (2) -- (3);
  \draw[<-, very thick, blue] (3) -- (4);
  \draw[->,very thick, blue] (3) -- (1);
  \draw[<-,very thick, blue] (2) -- (4);
  \node[below of=3, xshift=7.5mm, yshift=0.5cm] {(ii)};
  \end{scope}
 \begin{scope}[xshift = 6cm, yshift = 0cm]
  \node[rv]  (1)            {$1$};
  \node[rv, right of=1] (2) {$2$};
  \node[rv, below of=1] (3) {$3$};
  \node[rv, right of=3] (4) {$4$};
  \draw[<->, very thick, red] (1) -- (2);
  \draw[<->, very thick, red] (2) -- (3);
  \draw[<->, very thick, red] (3) -- (4);
  \draw[<->,very thick, red] (3) to (1);
  \draw[->,very thick, blue] (2) -- (4);
  \node[below of=3, xshift=7.5mm, yshift=0.5cm] {(iii)};
  \end{scope}
\end{tikzpicture}
 \caption{Three MAGs where (i) and (ii) are Markov equivalent but (iii) is not. }
 \label{MEF}
\end{figure}

\begin{table}[h]
\caption{Heads and tails of graphs in Figure \ref{MEF}}\label{tab1}
\begin{center}
\begin{tabular}{|c|c|c||c|c|c|} 
 \hline
  Figure & heads & tails & Figure & heads & tails \\ 
 \hline
 \multirow{6}{*}{\ref{MEF}(i)} & 1 & 3 & \multirow{10}{*}{\ref{MEF}(iii)} & 1 & $\emptyset$ \\
  & 2 & $\emptyset$ & & 2 & $\emptyset$\\
  & 3 & $\emptyset$ &  & 3 & $\emptyset$\\
  & 4 & 2,3 & &4 & 2\\
  & 1,2 & 3&  & 1,2 &$\emptyset$ \\
  & 2,3 & $\emptyset$ &  &1,3  & $\emptyset$\\
 \cline{1-3}
  \multirow{4}{*}{\ref{MEF}(ii)}&1  &2,3  &  & 2,3 & $\emptyset$\\  
  & 2 & 4 &  & 3,4 & 2\\
  & 3 & 2,4 &  & 1,2,3 & $\emptyset$\\
  & 4 & $\emptyset$ &  & 1,3,4 & 2\\
 \hline
\end{tabular}
\end{center}
\end{table}

\begin{table}[h]
\caption{Parametrizing set of graphs in Figure \ref{MEF}}\label{tab2}
\begin{center}
\begin{tabular}{|c|c|c|} 
  \hline
  Figure & parametrizing sets & missing sets \\ 
  \hline
  \multirow{4}{*}{\ref{MEF}(i)(ii)} & $\{1\},\{2\},\{3\},\{4\}$ & $\{1,4\}$\\
  & $\{1,2\},\{1,3\},\{2,3\}$ & $\{1,2,4\}$\\ 
  & $\{2,4\},\{3,4\}$ & $\{1,3,4\}$ \\
  &$\{1,2,3\},\{2,3,4\}$ & $\{1,2,3,4\}$\\
 \hline
  \multirow{5}{*}{\ref{MEF}(iii)}& $\{1\},\{2\},\{3\},\{4\}$ & $\{1,4\}$\\
  & $\{1,2\},\{1,3\},\{2,3\}$ & $\{1,2,4\}$\\
  & $\{2,4\},\{3,4\}$ & \\ 
  & $\{1,2,3\},\{1,3,4\},\{2,3,4\}$ & \\
  & $\{1,2,3,4\}$ & \\[1ex] 
 \hline
\end{tabular}
\end{center}
\end{table}

\begin{table}[h]
\caption{$\mathcal{S}_{3}$ and $\Tilde{\mathcal{S}}_{3}$ graphs in Figure \ref{MEF}}\label{tab3}
\begin{center}
\begin{tabular}{|c|c|c|} 
  \hline
  \multirow{1}{*}[-1pt]{Figure} & \multirow{1}{*}[-1pt]{$\mathcal{S}_{3}$} & \multirow{1}{*}[-1pt]{$\Tilde{\mathcal{S}}_{3}$}\\ [0.5ex] 
  \hline
  \multirow{3}{*}{\ref{MEF}(i)(ii)} & $\{1,2\},\{1,3\},\{2,3\}$ & $\{1,2\},\{1,3\}$\\
  & $\{2,4\},\{3,4\}$ & $\{2,3\}, \{2,4\}$ \\
  &$\{1,2,3\},\{2,3,4\}$ & $\{3,4\}$\\
 \hline
  \multirow{4}{*}{\ref{MEF}(iii)} & $\{1,2\},\{1,3\},\{2,3\}$ & $\{1,2\},\{1,3\}$\\
  & $\{2,4\},\{3,4\}$ & $\{2,3\}, \{2,4\}$\\ 
  & $\{1,2,3\},\{2,3,4\}$ & $\{3,4\}$\\
  & $\{1,3,4\}$ & $\{1,3,4\}$\\ [1ex] 
 \hline
\end{tabular}
\end{center}
\end{table}
In Figure \ref{MEF}, (i) is Markov equivalent to (ii) and they also have the same parametrizing sets; however, (iii) has a different parametrizing set and is not Markov equivalent to either (i) or (ii). In Figure \ref{MEF}(i) and (ii), $1 \perp_m 4\mid 2,3$ is the only m-separation while Figure \ref{MEF}(iii) encodes $1 \perp_m 4\mid 2$. Note that these conditional independences correspond precisely to these missing sets which are in the form $\{a,b\} \cup C'$ where $a \perp_m b\mid C \text{ and } C' \subseteq C$. Thus it is reasonable to conjecture that equivalent graphs should have the same parametrizing sets. It turns out that not only is this true, but in fact equivalence conditions can be refined even further and it is sufficient to consider $\mathcal{S}_{3}$ or $\Tilde{\mathcal{S}}_{3}$.

\section{MARKOV EQUIVALENCE} \label{sec:mainres}

\subsection{Previous Work}

The first theorem on Markov equivalence of MAGs is from \citet{Spirtes97apolynomial}.

\begin{theorem}\label{thm:1.4}
Two MAGs $\mathcal{G}_{1}$ and $\mathcal{G}_{2}$ are Markov equivalent if and only if (i) $\mathcal{G}_{1}$ and $\mathcal{G}_{2}$ have the same adjacencies, (ii) $\mathcal{G}_{1}$ and $\mathcal{G}_{2}$ have the same unshielded colliders and  (iii) if $\pi$ forms a discriminating path for $b$ in $\mathcal{G}_{1}$ and $\mathcal{G}_{2}$, then $b$ is a collider on the path $\pi$ in $\mathcal{G}_{1}$ if and only it is a collider on the path $\pi$ in $\mathcal{G}_{2}$.
\end{theorem}

For $x$ and $y$ nonadjacent, a \emph{discriminating path} $\pi = \langle x, q_{1},\ldots,q_{m}, b,y \rangle$, $m \geq 1$ for $b$, is a subgraph comprised of a collection of paths:
\begin{align*}
 x& \rqarrow q_{1}\leftrightarrow \cdots \leftrightarrow q_{i}\rightarrow y, \qquad 1 \leq i \leq m;\\ 
x&\rqarrow q_{1}\leftrightarrow \cdots \leftrightarrow q_{m} \lqarrow b \rqarrow y.
\end{align*}
\par
For example, $\langle 1,2,3,4 \rangle$ forms a discriminating path for 3 in both Figure \ref{MEF}(i) and (iii), but not (ii). The vertex 3 is a collider on the path in (iii) but not (i), so (i) and (iii) are not equivalent; however (i) and (ii) are equivalent.  In general, the cost of identifying all the discriminating paths is not polynomial in the number of vertices and edges. However, we will make use of Theorem \ref{thm:1.4} in later proofs.

\subsection{Markov Equivalence Of MAGs}\label{MEMAG}
We now present the main result of this paper.
\begin{theorem}\label{thm:1.1}
Let $\mathcal{G}_{1}$ and $\mathcal{G}_{2}$ be two MAGs. Then $\mathcal{G}_{1}$ and $\mathcal{G}_{2}$ are Markov equivalent if and only if $\mathcal{S}(\mathcal{G}_{1})=\mathcal{S}(\mathcal{G}_{2})$. 
\end{theorem}

Theorem \ref{thm:1.1} already provides a method to find equivalence between two MAGs by searching all the heads and corresponding tails, however, the number of heads is not polynomial in the size of the graph.

\begin{corollary}\label{cor:1.1.2}
Let $\mathcal{G}_{1}$ and $\mathcal{G}_{2}$ be two MAGs. Then $\mathcal{G}_{1}$ and $\mathcal{G}_{2}$ are Markov equivalent if and only if $\mathcal{S}_3(\mathcal{G}_1) = \mathcal{S}_3(\mathcal{G}_2)$.  
This in turn occurs if and only if 
$\Tilde{\mathcal{S}}_{3}(\mathcal{G}_{1})=\Tilde{\mathcal{S}}_{3}(\mathcal{G}_{2})$.
\end{corollary}
 
 The motivation for defining $\Tilde{\mathcal{S}}_{3}(\mathcal{G})$ is that we cannot obtain the same complexity if we allow triangles to be included, as in $\mathcal{S}_3$.  To see this, consider a complete bidirected graph with $e$ edges: this will require $O(e^3)$ operations to list all the triangles (which are all heads). 
%  For example, if a graph $\mathcal{G}$ has an m-separation $a \perp_m b,c \mid d$ then certainly $a \perp_m b \mid d$ so not only $\{a,b,c,d\} \notin \mathcal{S}(\mathcal{G})$ but also $\{a,b,d\} \notin \mathcal{S}(\mathcal{G})$.  
  Note we do not care about triples with three or zero adjacencies. Theorem \ref{thm:1.4} tells us that apart from adjacencies between pairs of vertices, and unshielded triples which lack one adjacency, we only need to find that for a discriminating path $\pi = \langle x,q_{1},\ldots,q_{m}, b,y \rangle$, whether $b$ is a collider on the path or not. Later we will show that $\{x,b,y\} \in \mathcal{S}(\mathcal{G})$ if and only if $b$ is a collider on $\pi$, and note that $b,y$ are adjacent but $x,y$ are not. 
 \par

%+++++++++++++++++++++++++++++++++++++++++++++++++++++++++++++++++++++++++++++++++++ 
Corollary $\ref{cor:1.1.2}$ is particularly important for identifying Markov equivalence. It not only allows the algorithm to run in polynomial time as we only need to check heads with size at most 3, but also accelerates it further as we do not need to find triples with full adjacencies or no adjacencies, nor to store lots of triangles from the dense part of the graph.
\par
To prove Theorem \ref{thm:1.1} and Corollary \ref{cor:1.1.2}, we first prove the following propositions.
\begin{proposition}\label{msepaandS}
Let $\mathcal{G}$ be a MAG with vertex set $\mathcal{V}$. For a set $W \subseteq \mathcal{V}$, $W \notin \mathcal{S}(\mathcal{G})$ if and only if there are two vertices $a,b$ in $W$ such that we can m-separate them by a set $C$ such that $a,b \notin C$ with $W \subseteq C \cup \{a,b\}$.
\end{proposition}
\begin{proof}
We prove an equivalent statement of this proposition, that is: $W \in \mathcal{S}(\mathcal{G})$ if and only if for any two vertices $a,b$ in $W$ we cannot m-separate them by a set $C$ such that $a,b \notin C$ with $W \subseteq C \cup \{a,b\}$.
\par
To prove $\Rightarrow$: if $W \in \mathcal{S}(\mathcal{G})$, then there is a nonempty subset $W' \subseteq W$ such that $W'$ is a head and $W \subseteq W' \cup \tail(W')$. Because $\tail(W') \subseteq \an(W')$, we have $W' \cup \tail(W') \subseteq \an(W')$. By definition of the heads and tails, any two vertices $a,b$ in $W \subseteq W' \cup \tail(W')$ are connected by a collider path where all the colliders are in $\an(W') \subseteq \an(C \cup \{a,b\})$. Let $d_{i}$, 1$\leq i \leq n$ be intermediate vertices in the path. Now if all of $d_{i}$ are ancestors of $C$ then this path m-connects $a$ and $b$. So some of $d_{i}$ are only ancestors of ${a,b}$. 
\par
Suppose there exists some $d_{i} \in \an_\mathcal{G}(a) \setminus \an_{\mathcal{G}}(C)$, let $d_{j}$ be the furthest one on path $\pi$ from $a$, so there exists a directed path $\pi':a\leftarrow\cdots\leftarrow d_{j}$ such that none of vertices in $\pi'$ after $a$ is an ancestor of $C$ and hence not in $C$. If all  $d_{k}$ after $d_{j}$ belong to $\an_{\mathcal{G}}(C)$ then we find a m-connecting path between $a$ and $b$: $a \leftarrow \cdots \leftarrow d_{j} \leftrightarrow \cdots \leftrightarrow d_{n} \lqarrow b$. If not, let $d_{m}$ be the first one after $d_{j}$ such that $d_{m} \in \an_{\mathcal{G}}(b) \setminus \an_{\mathcal{G}}(C)$ then again we find a m-connecting path between $a$ and $b$: $a \leftarrow \cdots \leftarrow d_{j} \leftrightarrow \cdots \leftrightarrow d_{m} \rightarrow \cdots \rightarrow b$.
\par
If all $d_{i} \notin \an_{\mathcal{G}}(C)$ are ancestors of $b$ then let $d_{j}$ be the closest one to $a$ in path $\pi$ which also leads to a m-connecting path between $a$ and $b$: $a \rqarrow d_{1} \leftrightarrow \cdots \leftrightarrow d_{j} \rightarrow \cdots \rightarrow b$. Hence in all cases any $a,b$ in $W$ are not m-separated given any $C \supseteq W \setminus \{a,b\}$. 
\par
To prove $\Leftarrow$: define $W'$ = $\barren(W)$. We claim that it is a head. Suppose it is not a head, by the definitions of a barren set and a head, $W'$ does not lie in a single district in $\mathcal{G}_{\an(W')}$. Let $D_{i} \subset W'$ index bidirected-connected components of $W'$ in $\an_{\mathcal{G}}(W')$ where $1 \leq i \leq m$. Clearly by assumption $m >1$, and now consider $D_{1}$ and $D_{2}$. For any edge in $\mathcal{G}_{\an(W')}$ which has an endpoint $a \in W'$, it is of the form $a \lqarrow$ by definition of a barren set, so if there is a collider path between $D_{1}$ and $D_{2}$, it would be a bidirected path which is a contradiction to the definition of $D_{1}$ and $D_{2}$. This means that any path in $\an_{\mathcal{G}}(W')$ between $D_{1}$ and $D_{2}$ contains at least one non-collider which is not in $W'$ and hence it is in $\an_{\mathcal{G}}(W') \setminus W'$. Thus for any two vertices in $D_{1}$ and $D_{2}$, given $\an_{\mathcal{G}}(W') \setminus W'$, they are m-separated in $\an_{\mathcal{G}}(W')$. Since $\an_{\mathcal{G}}(W')$ is ancestral, the m-separation also holds in the whole graph. Thus $W'$ is a head.
\par
By Remark 4.14 in \citet{Evans2014}, for any head $H$ we have $H \perp_m \an_{\mathcal{G}}(H) \setminus (H \cup \tail(H)) \mid \tail(H)$. Thus if $(W \setminus W')$ is not in $\tail(W')$, we can m-separate a vertex in $(W \setminus W') \setminus \tail(W')$ and a vertex in $W'$ given the remaining vertices in $\an_{\mathcal{G}}(W')$, which is a contradiction.
\end{proof}

\begin{proposition}\label{Prodiscriminatinggraph}
For a MAG $\mathcal{G}$, we have (i) any two vertices $a$ and $b$ are adjacent in $\mathcal{G}$ if and only if $\{a,b\} \in \mathcal{S}(\mathcal{G})$; (ii) for any unshielded triple $(a,b,c)$ in $\mathcal{G}$, $\{a,b,c\} \in \mathcal{S}(\mathcal{G})$ if and only if $b$ is a collider on the triple $(a,b,c)$; (iii) if $\pi$ forms a discriminating path for $b$ with two end vertices $x$ and $y$ in $\mathcal{G}$ then $\{x,b,y\} \in \mathcal{S}(\mathcal{G})$ if and only if $b$ is a collider on the path $\pi$.
\end{proposition}

\begin{proof}
For (i), by maximality, any two vertices $a$ and $b$ are adjacent in a MAG if and only if we can not m-separate them by a set $C$, hence by Proposition \ref{msepaandS} if and only if $\{a,b\} \in \mathcal{S}(\mathcal{G})$.
\par
For (ii), the only nonadjacent pair of vertices are $a,c$, for any set $C$ that m-seperates them, $b \notin C$ if and only if $b$ is a collider on the triple $(a,b,c)$, hence by Proposition \ref{msepaandS} if and only if $\{a,b,c\} \in \mathcal{S}(\mathcal{G})$.
\par
For (iii), if $x,b$ are not adjacent, then for any set that m-separates them, $y$ is not in the set, as the path $x \rqarrow q_{1} \leftrightarrow \cdots  \leftrightarrow q_{m} \lqarrow b$ would be m-connecting $x$ and $b$. Since $x,y$ are not adjacent, there exists some set $C$ such that $x \perp_m y \mid C$. From page 11 in \citet{ali2009}, we know that for any such $C$, $q_{i} \in C$ for all $i \leq n$ and $b$ is a collider if and only if $b \notin C$, hence by Proposition \ref{msepaandS} if and only if $\{x,b,y\} \in \mathcal{S}(\mathcal{G})$.
\end{proof}

Now we are able to prove Theorem $\ref{thm:1.1}$ and Corollary $\ref{cor:1.1.2}$

\begin{proof}[Proof of Theorem \ref{thm:1.1}]
($\Rightarrow$) Proposition \ref{msepaandS} ensures that missing sets in $\mathcal{S}(\mathcal{G})$ are only due to m-separations in graphs. But as Markov equivalence is characterized by m-separations, $\mathcal{S}(\mathcal{G}_{1})$ and $\mathcal{S}(\mathcal{G}_{2})$ in two equivalent MAGs $\mathcal{G}_{1}$ and $\mathcal{G}_{2}$ are the same. ($\Leftarrow$) Proposition \ref{Prodiscriminatinggraph} implies that any violation of conditions in Theorem \ref{thm:1.4} result in different $\mathcal{S}(\mathcal{G}_{1})$ and $\mathcal{S}(\mathcal{G}_{2})$. Hence if $\mathcal{S}(\mathcal{G}_{1})=\mathcal{S}(\mathcal{G}_{2})$, $\mathcal{G}_{1}$ is Markov equivalent to $\mathcal{G}_{2}$.
\end{proof}

\begin{proof}[Proof of Corollary \ref{cor:1.1.2}]
($\Rightarrow$) This follows from Theorem \ref{thm:1.1} and the fact that Markov equivalent MAGs have the same adjacencies. ($\Leftarrow$) This follows from the fact that in the proof the `if' part of Theorem \ref{thm:1.1}, 
we only consider sets in $\Tilde{\mathcal{S}}_{3}(\mathcal{G}_{1})$ and $\Tilde{\mathcal{S}_{3}}(\mathcal{G}_{2})$.
\end{proof}
\citet{frydenberg1990chain} 
gives conditions for when two DAGs are equivalent, i.e.\ if and only if they have the same adjacencies and unshielded colliders. DAGs are a subclass of MAGs so Corollary \ref{cor:1.1.2} also applies to them.  When $\mathcal{G}$ is just a DAG, $\Tilde{\mathcal{S}}_{3}(\mathcal{G})$ (and indeed $\mathcal{S}_3(\mathcal{G})$) contains the exact information of $\mathcal{G}$'s adjacencies and unshielded colliders. 
By Proposition \ref{Prodiscriminatinggraph}, $\{a,b\} \in \Tilde{\mathcal{S}}_{3}(\mathcal{G})$ if and only if $a, b$ are adjacent. And a triple is in $\Tilde{\mathcal{S}}_{3}(\mathcal{G})$ if and only if it is an unshielded collider; this is because in DAGs, heads are precisely the individual vertices, and the corresponding tails are their  parent sets. 

\subsection{Projection From ADMGs To MAGs}\label{OMESG}

\citet{richardson2002} give a projection that projects a DAG $\mathcal{G}$ with latent variables $L$ to a Markov equivalent MAG $\mathcal{G}^{m}$: (i) every pair of vertices $a,b \in \mathcal{V}$ in $\mathcal{G}$ that are connected by an \emph{inducing path} becomes adjacent in $\mathcal{G}^{m}$; (ii) an edge connecting $a,b$ in $\mathcal{G}^{m}$ is oriented as follows: if $a \in \an_{\mathcal{G}}(b)$ then $a \rightarrow b$; if $b \in \an_{\mathcal{G}}(a)$ then $b \rightarrow a$; if neither is the case, then $a \leftrightarrow b$. An \emph{inducing path} between $a,b$ is a path such that every collider in the path is in $\an(\{a,b\})$, and every noncollider is in $L$. Note if we already have an ADMG $\mathcal{G}$, we can apply the projection to $\mathcal{G}$ with no latent variable to construct the corresponding $\mathcal{G}^{m}$, so an inducing path in this case is just a collider path with every collider in $\an(\{a,b\})$. In addition, the projection preserves ancestral relations from the original graph. 
\par
To extend previous theorems to $\mathcal{G}$ we need following lemmas to link $\mathcal{G}$ and $\mathcal{G}^{m}$. 

\begin{lemma}\label{lem:2.1}
If $v,w$ are connected by a collider path $\pi_{1}$ in an ADMG $\mathcal{G}$ then they are connected by a collider path $\pi_{2}$ in $\mathcal{G}^{m}$ where $\pi_{2}$ uses a subset of the internal vertices of $\pi_{1}$. Also, if $\pi_{1}$ starts with $v \rightarrow$, so does $\pi_{2}$.
\end{lemma}
Lemma \ref{lem:2.1} is in analogue to Lemma 23 in \citet{Shpitser2018AcyclicLS}. Now we show heads and tails are preserved through the projection.

\begin{proposition}\label{lem:2.7}
If $\mathcal{G}$ is an ADMG,  $\mathcal{H}(\mathcal{G})=\mathcal{H}(\mathcal{G}^{m})$ and for every $H \in \mathcal{H}(\mathcal{G})$, $\tail_{\mathcal{G}}(H)=\tail_{\mathcal{G}^{m}}(H)$.
\end{proposition}
\begin{proof}
Suppose $H$ is a head in $\mathcal{G}$. Then it is bidirected-connected in $\mathcal{G}_{\an(H)}$, so by Lemma $\ref{lem:2.1}$ each bidirected path connecting vertices in $H$ is preserved as a collider path in  $\mathcal{G}_{\an(H)}^{m}$. Further as the projection preserves ancestral relation and $H = \barren(\an(H))$, each path is bidirected. Hence any head $H$ in $\mathcal{G}$ is a head in $\mathcal{G}^{m}$. By similar argument, we can see that for a head $H$ in $\mathcal{G}$, any $w \in \tail_{\mathcal{G}}(H)$ is in $\tail_{\mathcal{G}^{m}}(H)$.
\par
Suppose $H$ is a head in $\mathcal{G}^{m}$ so it is bidirected-connected in $\an(H)$ in $\mathcal{G}^{m}$. But each bidirected edge in $\mathcal{G}^{m}$ corresponds to a collider path in $\mathcal{G}$ with intermediate colliders in ancestors of endpoints; hence as the projection preserves ancestral relations, the path is bidirected. Therefore $H$ is also a head in $\mathcal{G}$. Note in general for any $v \leftrightarrow w$ in $\mathcal{G}^{m}$, there is a bidirected path between them in $\mathcal{G}$.
\par
Let $z \in \tail_{\mathcal{G}^{m}}(H)$ so there is a collider path $\pi$ between $z$ and $h \in H$ in $\mathcal{G}^{m}$ ending $\cdots\leftrightarrow h$. We know every bidirected edge in the path $\pi$ corresponds to a bidirected path in $\an(H)$ in $\mathcal{G}$. If the path $\pi$ begins with $z \leftrightarrow$ then $z$ is bidirected-connected to $h$ in $\an(H)$ so $z \in \tail_{\mathcal{G}}(H)$. If the path $\pi$ begins with $z \rightarrow w_{1}$ then in $\mathcal{G}$ we have a collider path between $z$ and $w_{1}$ in $\an(H)$, which ends with $\leftrightarrow w_{1}$. Thus $z$ is also in $\tail_{\mathcal{G}}(H)$.
\end{proof}
\par
Definitions of heads and tails are closely related to the projection of ADMGs. The next lemma allows us to project an ADMG to a Markov equivalent MAG in polynomial time. The algorithm is shown in next section. Let $\mathcal{G}$ be a ADMG and $\mathcal{G}^{m}$ be its projected MAG.
\begin{lemma}\label{projsgtmag}
Let $v,w$ be two vertices then (i) $v \rightarrow w$ in $\mathcal{G}^{m}$ if and only if $v \in \tail_{\mathcal{G}}(w)$ and (ii) $v \leftrightarrow w$ in $\mathcal{G}^{m}$ if and only if $\{v,w\} \in \mathcal{H}(\mathcal{G})$.
\end{lemma}
Since there is at most one edge between any two vertices in a MAG, if we know the tails of every vertex in $\mathcal{G}^{d}$ and every head of size 2, this is sufficient to construct $\mathcal{G}^{m}$. 
\par
Consider Figure \ref{MAGs}(i), this is an ADMG but not a MAG. Tails of 1, 2, 3, 4 are $\{4\}$, $\{3\}$, $\emptyset$, $\emptyset$, respectively. Heads of size 2 are $\{1,2\}$, $\{1,3\}$, $\{2,4\}$, $\{3,4\}$, hence a Markov equivalent MAG of Figure \ref{MAGs}(i) preserves all the original edges and adds one edge $1 \leftrightarrow 2$.

\subsection{Markov Equivalence Of ADMGs}
We now show that Theorem \ref{thm:1.1} and Corollary \ref{cor:1.1.2} can be extended to ADMGs. Note that in general two Markov equivalent ADMGs do not necessarily have the same adjacencies defined with respect to edges; thus we need to redefine adjacencies in terms of m-separations.

\begin{definition}
For a ADMG $\mathcal{G}$ and two vertices $v,w$ in $\mathcal{G}$, $v$ and $w$ are \emph{adjacent} if and only if there is no set $C$ such that $v \perp_m w \mid C$ with $v,w \notin C$.
\end{definition}

Two vertices that are connected by an edge are clearly adjacent, we are excluding pairs that do not share any edges and yet have no conditional independence. In maximal graphs, these two definitions are equivalent.

\begin{theorem}\label{thm:2.8}
For two ADMGs $\mathcal{G}_{1}$ and $\mathcal{G}_{2}$, they are Markov equivalent if and only $\mathcal{S}(\mathcal{G}_{1})=\mathcal{S}(\mathcal{G}_{2})$.
\end{theorem}

\begin{corollary}\label{equisummary}
Two ADMGs $\mathcal{G}_{1}$ and $\mathcal{G}_{2}$ are Markov equivalent if and only if ${\mathcal{S}}_{3}(\mathcal{G}_{1})={\mathcal{S}}_{3}(\mathcal{G}_{2})$, and this occurs if and only if $\Tilde{\mathcal{S}}_{3}(\mathcal{G}_{1})=\Tilde{\mathcal{S}}_{3}(\mathcal{G}_{2})$.
\end{corollary}

\section{ALGORITHM}\label{algorithm}
In this section, $n$, $e$ denote number of vertices and total edges, respectively.

\floatstyle{boxed} % Box...
\restylefloat{table}% ...figure environment contents.

\begin{table*}
\caption{Algorithm 1: obtain $\Tilde{\mathcal{S}}_{3}(\mathcal{G})$ for a MAG $\mathcal{G}$} \label{algo1}
\begin{center}
\begin{tabular}{ll}
    \textbf{Input}: & A MAG $\mathcal{G}(\mathcal{V},\mathcal{E})$ \\ 
    \textbf{Output}: & $\Tilde{\mathcal{S}}_{3}(\mathcal{G})$\\ 
    1 & $S=\emptyset$;\\
    2 & \textbf{for} each $v \in \mathcal{V}$:\\
    3 & \quad \textbf{obtain} $\an_\mathcal{G}(v) = \{v\} \cup \an_\mathcal{G}(\pa_\mathcal{G}(v))$\\
    4 & \quad  \textbf{for} each $w \in \pa_{\mathcal{G}}(v)$:\\
    5 & \quad \quad $S = S \cup \{v,w\}$;\\
    6 & \quad \textbf{for} each $z, w \in \pa_{\mathcal{G}}(v)$ with $z \neq w$ and $z$ is not adjacent to $w$:\\
    7 & \quad \quad $S = S \cup \{v,w,z\}$;\\
    8 & \textbf{for} each $v \leftrightarrow w$:\\
    9 & \quad $S = S \cup \{v,w\}$;\\
    10 & \quad $\tail(\{v,w\})=\dis_{\an(\{v,w\})}(v) \cup \pa_{\mathcal{G}}(\dis_{\an(\{v,w\})}(v)) \setminus \{v,w\}$;\\
    11 & \quad \textbf{for} each $z \in \tail(\{v,w\})$ with $z$ not adjacent to both $v$ and $w$:\\ 
    12 & \quad \quad $S = S \cup \{v,w,z\}$;\\
    13 & \quad \textbf{for} each $z \in \sib_{\mathcal{G}}(\an_{\mathcal{G}}(\{v,w\})) \cap \dis_{\mathcal{G}}(v) \setminus (\an_{\mathcal{G}}(\{v,w\}) \cup \de_{\mathcal{G}}(\{v,w\})$\\
    14 & \quad \quad \quad and not adjacent to both $v$ and $w$:\\
    15 & \quad \quad obtain $\dis_{\an(\{v,w,z\})}(v)$;\\
    16 & \quad \quad \textbf{if} $z \in \dis_{\an(\{v,w,z\})}(v)$:\\ 
    17 & \quad \quad $S = S \cup \{v,w,z\}$;\\
    18 & \textbf{return} $S$\\

\end{tabular}
\end{center}
\end{table*}

\begin{table*}
\caption{Algorithm 2: obtain a MAG $\mathcal{G}^{m}$ for an ADMG $\mathcal{G}$} \label{algo2}
\begin{center}
\begin{tabular}{ll}
    \textbf{Input}: & An ADMG $\mathcal{G}(\mathcal{V},\mathcal{E})$ \\ 
    \textbf{Output}: & A Markov equivalent MAG $\mathcal{G}^{m}(\mathcal{V}, \mathcal{E}^m)$\\ 
    1 & Start with $\mathcal{G}^{m}$ that have the same vertices as $\mathcal{G} $ but no adjacencies;\\
    2 & \textbf{for} each $v \in \mathcal{V}$:\\
    3 & \quad \textbf{obtain} $\an_\mathcal{G}(v) = \{v\} \cup \an_\mathcal{G}(\pa_\mathcal{G}(v))$\\
    4 & \quad $\tail(v)=\dis_{\an(v)}(v) \cup \pa_{\mathcal{G}}(\dis_{\an(v)}(v)) \setminus \{v\}$\\
    5 & \quad \textbf{add} $w \rightarrow v \in \mathcal{E}^{m}$ for each $w \in \tail(v)$;\\
    6 & \textbf{for} each $v,w \in \mathcal{V}$ with no ancestral relation and in the same district:\\
    7 & \quad obtain $\dis_{\an(\{v,w\})}(v)$;\\
    8 & \quad \textbf{if} $w \in \dis_{\an(\{v,w\})}(v)$:\\
    9 & \quad \quad \textbf{add} $v \leftrightarrow w \in \mathcal{E}^{m}$;\\
    10 & \textbf{return} $\mathcal{G}^{m}$\\
\end{tabular}
\end{center}
\end{table*}

\subsection{MAGs}

 We assume that $n=O(e)$, since otherwise the graph will be disconnected.  Firstly, we propose an algorithm to identify $\Tilde{\mathcal{S}}_{3}(\mathcal{G})$ of a given MAG $\mathcal{G}$ and show that it runs in polynomial time ($O(ne^{2})$). To test equivalence of two MAGs, it is sufficient to compare their $\Tilde{\mathcal{S}}_{3}$, by Corollary \ref{cor:1.1.2}. Vertices are assumed to be in topological order. If not, this can be achieved with an $O(n+e)$ sort. We assume we have access to $\pa_\mathcal{G}(v)$ and $\sib_\mathcal{G}(v)$ for each $v \in \mathcal{V}$.

Let $A_1(\mathcal{G})$ denote the output of Algorithm 1 when applied to a MAG, $\mathcal{G}$. 

\begin{proposition}\label{A}
For a MAG $\mathcal{G}$, $A_{1}(\mathcal{G}) = \Tilde{\mathcal{S}}_{3}(\mathcal{G})$.
\end{proposition}

\subsection{Complexity Of Algorithm 1}

    The first loop from line 2 to line 7 runs at most $O(e^{2})$ times as the worst case is that one vertex have all others as its parents. There are at most $e$ bidirected edges so the second loop from line 8 to line 17 repeats at most $e$ times. There are three esrial tasks inside the second loop. The first one is line 10 which obtains the tails of $\{v,w\}$. The computation of obtaining tails given parents is $O(n+e)$. The second task, i.e.\ the first subloop from line 11 to line 12, is carried at most $n-2$ times as the size of each tail is at most $n-2$. For the third task from line 13 to line 17, there are at most $n-2$ potential candidates for the third member, and obtaining the district costs $O(n+e)$. Thus the overall complexity of Algorithm 1 is $O(e^{2}+e((n+e)+n+n(n+e)))=O(ne^{2})$. 
    \par
    Note that the number of potential candidates for third member of heads of size 3 depends on sizes of districts. If the number is high then it means districts are large so there are at least as many bidirected edges as potential candidates, so if the graph is sparse we can use $e$ to represent the number of candidates instead of $n$ when computing complexity. There are most $O(e^{2})$ sets in $\Tilde{\mathcal{S}}_{3}(\mathcal{G})$, and some graphs achieve this bound, for example, a DAG where one vertex have all others as its parents. 
    \par
    To test ordinary Markov equivalence of two MAGs, it is sufficient to compare their output of Algorithm 1 after a sort of order $O(e^{2} \log e^{2}) = O(e^{2} \log e)$. Note that $\log e = O(\log n)$, therefore the complexity of verifying Markov equivalence between two MAGs is still $O(ne^{2})$. Thus our algorithm is faster than the one proposed by \citet{ali2009}, which is only $O(ne^{4})$. 

\subsection{ADMGs}
 Algorithm 2 converts an ADMG $\mathcal{G}$ to a Markov equivalent MAG $\mathcal{G}^{m}$, as proven by Lemma \ref{projsgtmag}. To test Markov equivalence between two ADMGs, it is sufficient to put their equivalent MAGs in Algorithm 1 to obtain the corresponding sets $\Tilde{\mathcal{S}}_{3}$ and compare the sets.

\subsection{Complexity Of Algorithm 2}
For the first loop from line 3 to 5, it costs $O(n(n+e))$ since there are $n$ vertices and it takes $O(n+e)$ to obtain a district. The second loop from line 6 to 9 is at $O(n^{2}(n+e))$. Thus the overall complexity is $O(n(n+e)+n^{2}(n+e))=O(n^{3}+n^{2}e)=O(n^{2}e)$. The total cost for identifying Markov equivalence between two ADMGs is therefore $O(ne^{2})$.

\subsection{Comparison To Previous Algorithms}
Among previous characterizations of MAGs, only \citet{ali2009} provide a polynomial time algorithm to verify Markov equivalence. They consider all triples in a discriminating path; in order to do this, they iterate through (up to) $n-2$ levels; at each level they consider all remaining colliders ($O(e^2)$) and then check each set of reachable edges ($O(e^2)$). Conversely, we ignore any triples for which all three adjacencies are present (since they will trivially always be present). 
\par
In addition to the reduction in complexity, if we modify Algorithm 1 to compute % $\mathcal{S}(\mathcal{G})$ or 
$\mathcal{S}_{3}(\mathcal{G})$, the output contains more information. By Proposition \ref{msepaandS}, a set $\{a_1, a_2, a_3\}$ is missing from $\mathcal{S}_3$ if and only if there is a corresponding m-separation between (say) $a_1, a_2$ conditional on a set that includes $a_3$.  Thus we can view the parametrizing set as a summary of independence information in the graph. This is a novel perspective compared to previous theorems, which characterize graphs by structures like minimal collider paths or colliders with order, and do not have a straightforward connection to conditional independence.

\subsection{Empirical Complexity}

An experiment on random graphs shows that empirical complexity of Algorithm 1 is at $O(e^2)$ for many sparse graphs ($e = O(n)$). One random graph (ADMG) is generated in the following way. We first fix a topological ordering and the total number of edges ($e = 3n$). Then two vertices become adjacent with uniform probability. Once skeleton is determined, an edge is independently either directed or bidirected with $p = 0.5$. For each $n = 20,40,60,80,100$, we generate $N = 250$ random graphs then average the empirical complexity. Figure \ref{empiricalcomplexity} is the empirical complexity against $n^2$.

\begin{figure}
  \centering
  \includegraphics[scale = 0.4]{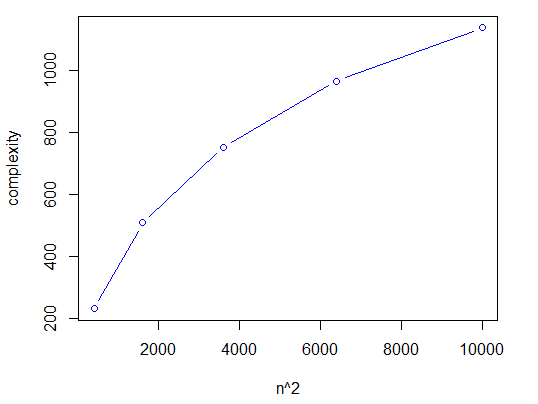} 
  \caption{Empirical complexity against $n^{2}$}
  \label{empiricalcomplexity}
\end{figure}

Suppose directed edges are added independently with probability $r/n$ according to a predetermined topological order, where $n$ is the number of vertices and $r \in \mathbb{R}^+$ is constant. The following proposition bounds the size of the ancestor sets in our sparse random graphs.  In particular, the largest average number of ancestors is at most $e^r$.
 
 \begin{proposition}
Let $A_i$ be the number of ancestors of the vertex $i$.  Then
\begin{align*}
\mathbb{E} A_i = \left( 1 + \frac{r}{n} \right)^{i-1}.
\end{align*}
In particular,
\begin{align*}
\mathbb{E} A_n = \left( 1 + \frac{r}{n} \right)^{n-1} \longrightarrow e^r.
\end{align*}
\end{proposition}

Markov's inequality gives us an easy corollary.

\begin{corollary}
$\mathbb{P}(A_i \geq k) \leq e^r/k$ for any $1 \leq i \leq n$ and $k \geq 1$.
\end{corollary}

Now it is straightforward to show that for sparse graphs, the complexity will be $O(e^2)$. This is because the main contribution of the complexity comes from counting heads of size 3. By bounding the sizes of ancestor sets, line 15 will run in constant time $O(1)$ instead of $O(n+e)$. Thus the overall complexity for sparse graphs is at $O(e^{2}+e((n+e)+n+n))=O(e^{2})$.

Here is an example for which the upper bound of complexity of Algorithm 1 is reached. Consider the graph in Figure \ref{complexity}. For every $i$ and $j$, $\{v_{i}, w, z_{j}\}$ forms a head of size 3. If $N,M,L$ are at $O(n)$ then the cost for identifying all these heads is at $O(ne^2)$.
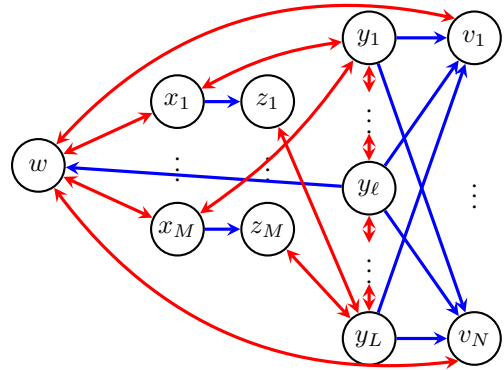
\begin{figure}[ht]
\begin{center}
   \begin{tikzpicture}
   [rv/.style={circle, draw, thick, minimum size=7mm, inner sep=0.5mm}, node distance=12mm, >=stealth]
   \pgfsetarrows{latex-latex};
 \begin{scope}
   \node[rv]  (w)       {$w$};
   \node[rv, above right of=w, xshift=10mm]  (x1)       {$x_1$};
   \node[below of=x1, yshift=4mm] (dots2) {$\vdots$};
   \node[rv, below right of=w, xshift=10mm]  (xM)       {$x_M$};
   \node[rv, right of=x1]  (z1)  {$z_1$};
   \node[right of=dots2] (dots3) {$\vdots$};
   \node[rv, right of=xM]  (zM)  {$z_M$};
   \node[rv, above right of=z1, xshift=5mm] (y1) {$y_1$};
   \node[inner sep=0mm, below of=y1, yshift=2mm] (dots4) {$\vdots$};
   \node[rv, below of=y1, yshift=-8mm] (yl) {$y_\ell$};
   \node[inner sep=0mm, below of=yl, yshift=2mm] (dots5) {$\vdots$};
   \node[rv, below of=dots5, yshift=2mm] (yL) {$y_L$};
      \node[rv, right of=y1, xshift=2mm]  (v1)       {$v_1$};
   \node[right of=yl, xshift=2mm] (dots) {$\vdots$};
   \node[rv, right of=yL, xshift=2mm]  (vN)       {$v_N$};
   \draw[<->, very thick, red] (v1.135) to[bend right=30] (w);
   \draw[<->, very thick, red] (vN.235) to[bend left=30] (w);
   \draw[<->, very thick, red] (x1) -- (w);
   \draw[<->, very thick, red] (xM) -- (w);
   \draw[->, very thick, blue] (x1) -- (z1);
   \draw[->, very thick, blue] (xM) -- (zM);
   \draw[<->, very thick, red] (y1) -- (dots4);
   \draw[<->, very thick, red] (yl) -- (dots4);
   \draw[<->, very thick, red] (yl) -- (dots5);
   \draw[<->, very thick, red] (yL) -- (dots5);
   \draw[->, very thick, blue] (y1) to[bend right=0] (v1);
   \draw[->, very thick, blue] (yl) to (v1);
   \draw[->, very thick, blue] (yL) to[bend right=0] (v1);
   \draw[->, very thick, blue] (yl) -- (w);
   \draw[->, very thick, blue] (y1) to[bend left=0] (vN);
  \draw[->, very thick, blue] (yl) -- (vN);
  \draw[->, very thick, blue] (yL) to[bend left=0 ] (vN);
  \draw[<->, very thick, red] (yL) to[bend left=0] (z1);
  \draw[<->, very thick, red] (yL) to[bend left=0] (zM);
  \draw[<->, very thick, red] (y1) to[bend right=10] (x1);
  \draw[<->, very thick, red] (y1) to[bend left=10] (xM);
   \end{scope}
 \end{tikzpicture}
 \end{center}
  \caption{A sequence of graphs in which the maximum complexity is achieved by Algorithm 1.  Note that $y_1$ is connected by a bidirected edge to every $x_i$, and $y_L$ to every $z_i$.}
 \label{complexity}
 \end{figure}
 
% This is verified by Figure \ref{ancestors} which shows quantiles 
% converging to steady levels for the last 20 vertices in large graphs.

% \begin{figure}
%     \centering
%   \includegraphics[scale = 0.4]{uai_anc.pdf}
%     \caption{Quantiles for numbers of ancestors against $n$ (log-scale).}
%     \label{ancestors}
% \end{figure}

% \input{thoughts.tex}

% \input{Discussion.tex}

% Acknowledgements should go at the end, before appendices and references

% \subsection*{Acknowledgements}
% We thank reviewers for giving advice on studying complexity of Algorithm 1.

\begin{appendices}

\section{Proofs from Section 3}

\begin{lemma35}
If $v,w$ are connected by a collider path $\pi_{1}$ in an ADMG $\mathcal{G}$ then they are connected by a collider path $\pi_{2}$ in $\mathcal{G}^{m}$ where $\pi_{2}$ uses a subset of the internal vertices of $\pi_{1}$. Also, if $\pi_{1}$ starts with $v \rightarrow$, so does $\pi_{2}$.
\end{lemma35}

\begin{proof}
Any adjacent pair in $\mathcal{G}$ is also adjacent in $\mathcal{G}^{m}$ as any edge is a trivial collider path. So the path $\pi_{1}$ is still present in $\mathcal{G}^{m}$ however it may not be a collider path (if it is then we are done) and we aim to find a collider path $\pi_{2}$. 
\par
Suppose $a$ is an internal vertex and is a noncollider in $\pi_{1}$ in $\mathcal{G}^{m}$ where $a \leftrightarrow b$ in $\mathcal{G}$ is changed to $a \rightarrow b$ in $\mathcal{G}^{m}$. This is because $a \in \an_{\mathcal{G}}(b)$. Consider the vertex $c$ on the other side of $a$, suppose it is $c \leftrightarrow a$ in $\mathcal{G}$. Then $b \leftrightarrow a \leftrightarrow c$ is a collider path where $a \in \an_{\mathcal{G}}(\{b,c\})$ so $b,c$ becomes adjacent in $\mathcal{G}^{m}$ and we can remove $a$ from the path. If $c \rightarrow a$, i.e.\ $c$ is one of end vertices, then in the projected graph we have $c \rightarrow a$. We can do this repeatedly until it terminates and the final path is a collider path in $\mathcal{G}^{m}$ that connects $v,w$.
\end{proof}

\begin{lemma37}
Let $v,w$ be two vertices then (i) $v \rightarrow w$ in $\mathcal{G}^{m}$ if and only if $v \in \tail_{\mathcal{G}}(w)$ and (ii) $v \leftrightarrow w$ in $\mathcal{G}^{m}$ if and only if $\{v,w\} \in \mathcal{H}(\mathcal{G})$.
\end{lemma37}

\begin{proof}
For (i), if $v \rightarrow w$ in $\mathcal{G}^{m}$ then $v \in \an_{\mathcal{G}}(w)$ and in $\mathcal{G}$ there is an inducing path between $v$ and $w$ (a collider path). If $v \rightarrow w$ in $\mathcal{G}$ then we are done. Otherwise any intermediate vertex on the path is in $\an_{\mathcal{G}}(\{v,w\}) = \an_{\mathcal{G}}(w)$ hence $v \rqarrow ... \leftrightarrow w$. Therefore $v \in \tail_{\mathcal{G}}(w)$. Conversely, $v \in \tail_{\mathcal{G}}(w)$ implies that $v \in \an_{\mathcal{G}}(w)$ and there is a collider path between $v$ and $w$ with any intermediate vertex in $\an_{\mathcal{G}}(w)$ hence the path is an inducing path and $v \rightarrow w$ in $\mathcal{G}^{m}$.
\par
For (ii), if $v \leftrightarrow w$ in $\mathcal{G}^{m}$ then there is an inducing path between $v$ and $w$ (a collider path) in $\mathcal{G}$ and $v,w$ are not ancestors to each other. Also any intermediate vertex on the path is in $\an_{\mathcal{G}}(\{v,w\})$ which suggests that the path is a bidirected path. Therefore, $\{v,w\}$ forms a head. On the other hand, if $\{v,w\}$ is a head in $\mathcal{G}$ then they are not ancestors to each other and there is a bidirected path between them with any intermediate vertex in $\an_{\mathcal{G}}(\{v,w\})$ so this path is an inducing path and $v \leftrightarrow w$ in $\mathcal{G}^{m}$.
\end{proof}

\begin{theorem38}
For two ADMGs $\mathcal{G}_{1}$ and $\mathcal{G}_{2}$, they are ordinary Markov equivalent if and only $\mathcal{S}(\mathcal{G}_{1})=\mathcal{S}(\mathcal{G}_{2})$.
\end{theorem38}

\begin{proof}
This follows from Proposition 3.7 and Theorem 3.2.
\end{proof}

\begin{corollary381}
Two ADMGs $\mathcal{G}_{1}$ and $\mathcal{G}_{2}$ are ordinary Markov equivalent if and only ${\mathcal{S}}_{3}(\mathcal{G}_{1})={\mathcal{S}}_{3}(\mathcal{G}_{2})$, and this occurs if and only if $\Tilde{\mathcal{S}}_{3}(\mathcal{G}_{1})=\Tilde{\mathcal{S}}_{3}(\mathcal{G}_{2})$.
\end{corollary381}

\begin{proof}
By Proposition 3.8, $\mathcal{S}_{3}$ are preserved in $(\mathcal{G}_{1})^{m}$ and $(\mathcal{G}_{2})^{m}$, and with the new definition of adjacencies, the outputs of  $\Tilde{\mathcal{S}}_{3}$ are also preserved. Hence the statement follows from Corollary 3.2.1.
\end{proof}
\section{Extension to Summary Graphs and MAGs with undirected edges}
MAGs defined in \citet{richardson2002} contain undirected edges which necessitate additional conditions of ancestrality. In addition to the previous condition ($\sib_{\mathcal{G}}(v) \cap \an_{\mathcal{G}}(v) = \emptyset$ and this is referred as condition 1 of ancestrality), one also requires that if an undirected edge is present between two vertices $v$ and $w$ then there is no arrow into $v$ or $w$. We refer to this as condition 2 of ancestrality. 

\begin{definition}
A graph $\mathcal{G}$ is \emph{ancestral} if: (1) for every $v \in \mathcal{V}$, $\sib_{\mathcal{G}}(v) \cap \an_{\mathcal{G}}(v) = \emptyset$; (2) if there is an undirected edge $x-y$ then $x,y$ have no parents and no siblings.
\end{definition}

A direct consequence of this definition is that vertices with undirected edges are `at the top' of the graph $\mathcal{G}$. For an acyclic graph $\mathcal{G}$ with three types of edges and only satisfying condition 2 of ancestrality, it can be seen as an ADMG with an undirected component among vertices without parents or siblings and therefore the component is "at the top" of the graph.
\par
\emph{Summary graphs} defined in \citet{wermuth2011} are actually the same as ADMGs with undirected components at the top. Graphically, one just needs to change the dashed lines to bidirected edges and they encode the same conditional independence. For simplicity, we will refer to this type of graphs as summary graphs.  Among the three graphs in Figure \ref{summarygraph}, (ii) is the only summary graph.
\begin{definition}
For a summary graph $\mathcal{G}$, let $U = \{v \in \mathcal{V}:v-w \text{ for some }w \in \mathcal{V}\}$ and $D = \mathcal{V} \setminus U$. Define $\mathcal{G}^{u} = \mathcal{G}_{U}$ and $\mathcal{G}^{d} = \mathcal{G}_{D}$.
\end{definition}
It is showed by \citet{richardson2002} that we can always split a summary graph into two disjoint subgraphs. One is an undirected subgraph $\mathcal{G}^{u}$ and another one is a subgraph with only directed and bidirected edges $\mathcal{G}^{d}$. Note that heads and barren sets are only defined in $\mathcal{G}^{d}$, and tails may include vertices in both $\mathcal{G}^{u}$ and $\mathcal{G}^{d}$. 
\par
For example, Figure \ref{summarygraph}(ii) can be split as $\mathcal{G}^{u} = 1-2$ and $\mathcal{G}^{d} = 3 \leftarrow 4, 3 \leftrightarrow 4$. Its heads are $\{3\}$ and $\{4\}$ and the corresponding tails are $\{2,4\}$ and $\{2\}$. 
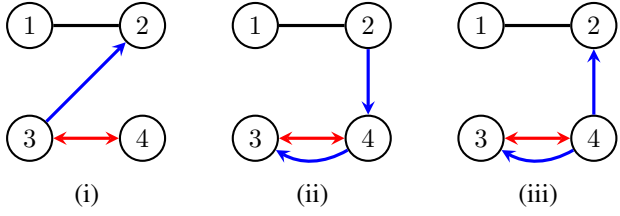
\begin{figure}
  \begin{tikzpicture}
  [rv/.style={circle, draw, thick, minimum size=6mm, inner sep=0.8mm}, node distance=15mm, >=stealth]
  \pgfsetarrows{latex-latex};
\begin{scope}
  \node[rv]  (1)            {$1$};
  \node[rv, right of=1] (2) {$2$};
  \node[rv, below of=1] (3) {$3$};
  \node[rv, right of=3] (4) {$4$};
  \draw[-, very thick] (1) -- (2);
  \draw[<->, very thick, red] (3) -- (4);
  \draw[->, very thick, color=blue] (3) -- (2);
  \node[below right of=3,xshift=-0.3cm, yshift=0.3cm] {(i)};
  \end{scope}
\begin{scope}[xshift = 3cm]
  \node[rv]  (1)            {$1$};
  \node[rv, right of=1] (2) {$2$};
  \node[rv, below of=1] (3) {$3$};
  \node[rv, right of=3] (4) {$4$};
  \draw[-, very thick] (1) -- (2);
  \draw[<->, very thick, red] (3) -- (4);
  \draw[->, very thick, color=blue] (2) -- (4);
  \draw[->,very thick, blue] (4) to [bend left=30] (3);
  \node[below right of=3,xshift=-0.3cm, yshift=0.3cm] {(ii)};
  \end{scope}
 \begin{scope}[xshift = 6cm]
  \node[rv]  (1)            {$1$};
  \node[rv, right of=1] (2) {$2$};
  \node[rv, below of=1] (3) {$3$};
  \node[rv, right of=3] (4) {$4$};
  \draw[-, very thick] (1) -- (2);
  \draw[<->, very thick, red] (3) -- (4);
  \draw[->, very thick, color=blue] (4) -- (2);
  \draw[->,very thick, blue] (4) to [bend left=30] (3);
  \node[below right of=3,xshift=-0.3cm, yshift=0.3cm] {(iii)};
  \end{scope}
\end{tikzpicture}
 \caption{(i) A graph that satisfies only condition 1 of ancestrality. (ii) A graph that satisfies only condition 2 of ancestrality. (iii) A graph that does not satisfy either condition 1 or 2 of ancestrality. }
 \label{summarygraph}
\end{figure}
\par
A vertex $a$ is said to be \emph{anterior} to $b$ if there is a path $\pi$ on which every edge is either undirected or directed towards $b$, or if $a=b$.  We denote the collection of all vertices anterior to $b$ by $\ant_{\mathcal{G}}(b)$.
\par
An \emph{undirected graph} (UG) is a graph with only undirected edges. A \emph{clique} in an UG is defined as a complete subset of vertices, that is: every pair of vertices is connected by an undirected edge. 
\par
For summary graphs, including MAGs, a \emph{clique} is defined in the same manner for vertices in $\mathcal{G}^u$, with completeness referring only to adjacencies by undirected edges. 
\begin{remark}
We extend the definition of parametrizing set by adding all the cliques to the set.
\end{remark}

\subsection{Extension to MAGs With Undirected Edges}

We only need to add a few line of argument to extend previous propositions and theorems.
\par
For $\Rightarrow$ of Proposition 3.3: if $W \in \mathcal{S}(\mathcal{G})$, then either $W$ is a clique or there is a nonempty subset $W' \subseteq W$ such that $W'$ is a head and $W \subseteq W' \cup \tail(W')$. The latter case is proved in the main paper. For the former case, it clearly implies that we can not m-separate any two vertices in $W$, given the remaining vertices in $W$.
\par
For $\Leftarrow$ of Proposition 3.3: For $W$ that does not lie entirely in $\mathcal{G}^{u}$ we can define $W'$ = $\barren(W)$. For $W$ lying in $\mathcal{G}^{u}$, if we cannot m-separate any two vertices in $W$ then clearly $W$ is a clique and $W \in \mathcal{S}(\mathcal{G})$.
\par
Proposition 3.4 does not change if we add undirected edges in MAGs, thus Theorem 3.2 and Corollary 3.2.1 hold for MAGs with undirected edges.

\subsection{Extension to Summary Graphs}
The projection described in Section 3.3 can be extend to summary graphs with latent variables $L$ as stated in \citet{richardson2002}. The modified projection is: (i) every pair of vertices $a,b \in \mathcal{V}$ in $\mathcal{G}$ that are connected by an \emph{inducing path} becomes adjacent in $\mathcal{G}^{m}$; (ii) an edge connecting $a,b$ in $\mathcal{G}^{m}$ is oriented as follows: if $a \in \ant_{\mathcal{G}}(b)$ then $a \rightarrow b$; if $b \in \ant_{\mathcal{G}}(a)$ then $b \rightarrow a$; if neither is the case, then $a \leftrightarrow b$; if they are both anterior to one another then the edge is undirected. An \emph{inducing path} between $a,b$ is a path such that every collider in the path is in $\an(\{a,b\})$, and every noncollider is in $L$. Again, we only consider projections with no latent variable, so an inducing path is just a collider path with every collider in $\an(\{a,b\})$. And the projection still preserves ancestral relations from the original graph. We first show that undirected edges are preserved through projections.
\begin{lemma}\label{unpreserv}
If $\mathcal{G}$ is a summary graph and $\mathcal{G}^{m}$ is its corresponding projected MAG,  then $\mathcal{G}^{u}=(\mathcal{G}^{m})^{u}$ and $(\mathcal{G}^{d})^{m}=(\mathcal{G}^{m})^{d}$.
\end{lemma}
\begin{proof}
For the first statement, we can prove it by showing that undirected edges are the same. First of all, notice that all undirected edges in $\mathcal{G}$ is preserved in $\mathcal{G}^{m}$. Secondly, no additional undirected edges can be added. If $a$ and $b$ are both in $\mathcal{G}^{u}$ then if they are not adjacent before, they are still nonadjacent since there is no inducing path between them (they are already at the top of the graph). If $a$ and $b$ are both in $\mathcal{G}^{d}$ then they cannot be anterior to each other, this would violate condition (ii) of ancestrality or the fact that $\mathcal{G}$ is acyclic. If $a \in \mathcal{G}^{u}$ and $b \in \mathcal{G}^{d}$ then obviously $b$ cannot be anterior to $a$.
\par
For the second statement, note the two subgraphs have the same vertices due to the first statement. For vertices in $\mathcal{G}^{d}$, ancestral relations are the same in $\mathcal{G}$ as there is no directed path passing $\mathcal{G}^{u}$. Also when we consider inducing paths, any such path would not contain any vertex in $\mathcal{G}^{u}$.
\end{proof}

We now show that Proposition 3.6 also holds for summary graphs, i.e.\ heads and tails are preserved through projection. 

\begin{proof}
So we have proved that for ADMGs, heads and tails are preserved through the projection. Now heads are only defined in $\mathcal{G}^{d}$ and $(\mathcal{G}^{m})^{d}$, thus by Lemma $\ref{unpreserv}$, for a summary graph, heads are preserved in $\mathcal{G}^{m}$. Also for tails that are in $\mathcal{G}^{d}$, they are preserved. It remains to show that the result holds when tails are in $\mathcal{G}^{u}$. For a head $H$, let $w \in \mathcal{G}^{u}$. If $w \in \tail_{\mathcal{G}}(H)$ then we know there is a path $\pi: w \rightarrow w_{1} \leftrightarrow \cdots \leftrightarrow h$, for $h \in H$ with intermediate vertices in $\an(H)$. Although $w \notin \mathcal{G}^{d}$, with the same argument in Lemma 3.5, this path is preserved as a collider path in $\an(H)$ in $\mathcal{G}^{m}$ with $\leftrightarrow h$ ($h$ is in a head) hence $w \in \tail_{\mathcal{G}^{m}}(H)$. Suppose now $w \in \tail_{\mathcal{G}^{m}}(H)$, so there is a path $\pi: w \rightarrow w_{1} \leftrightarrow \cdots \leftrightarrow h$ with intermediate vertices in $\an(H)$, we know every bidirected edge corresponds to a bidirected path in $\an(H)$ in $\mathcal{G}$, and the first directed edge correspond to a path $\pi': w \rightarrow w_{1} \leftrightarrow \cdots \leftrightarrow w_{2}$ in $\mathcal{G}$ with intermediate vertices in $\an(w_{2}) \subseteq \an(H)$, thus $w \in \tail_{\mathcal{G}}(H)$.
\end{proof}

Since Proposition 3.6 holds for summary graphs, if we change the definition of \emph{adjacencies} in summary graphs in the same manner as ADMGs by referring to m-separations, Theorem 3.8 and Corollary 3.8.1 also hold for summary graphs.

\subsection{Extension for Algorithms}

For Algorithm 1, we only add a line at the end of the algorithm (after line 17) to obtain the connected pairs in $\mathcal{G}^{u}$ (referred as line 18 in the next section). This costs $O(e)$ and hence does not contribute to the overall complexity. 

For Algorithm 2, as showed by Lemma \ref{unpreserv}, undirected edges are preserved, it is sufficient to add a line at the end of the algorithm (after line 9) to keep all the undirected edges. This costs $O(e)$ and hence does not contribute to the overall complexity. 

\section{Proof that Algorithm 1 outputs $\Tilde{\mathcal{S}}_3$}

Let $A_{1}(\mathcal{G})$ be the output of Algorithm 1 and $A'_{1}(\mathcal{G})$ be the output of Algorithm 1 without checking adjacencies in lines 6, 11 and 14. We also define the following sets for a MAG $\mathcal{G}$:
\begin{align*}
H_{1}(\mathcal{G}) &= \{\{v,w,z\}:v \in \mathcal{V} \text{ and }w,z \in \pa_{\mathcal{G}}(v)\}\\
H_{2}(\mathcal{G}) &= \{\{v,w,z\}:v \leftrightarrow w, z \in \tail(\{v,w\})\}\\
H^{a}_{3}(\mathcal{G}) &= \text{all heads of size 3 with some adjacencies}\\
H^{n}_{3}(\mathcal{G}) &= \text{all heads of size 3 with no adjacencies}\\
H_{3}(\mathcal{G}) &= \text{all heads of size 3} = H^{a}_{3}(\mathcal{G}) \cup H^{n}_{3}(\mathcal{G})\\
\hat{\mathcal{S}}_{3}(\mathcal{G}) &= \{S \in \mathcal{S}_{3}(\mathcal{G}):\text{there are some adjacencies in }S\}\\
U_{3}(\mathcal{G}) &= \{S \subseteq \mathcal{V}(\mathcal{G}^{u}):\abs{S} = 3 \text{ and }S\text{ is complete}\}.
\end{align*}
Thus by definition $\Tilde{\mathcal{S}}_{3}(\mathcal{G}) \subseteq \hat{\mathcal{S}}_{3}(\mathcal{G}) \subseteq \mathcal{S}_{3}(\mathcal{G})$ and $\mathcal{S}_{2}(\mathcal{G})$, $H_{1}(\mathcal{G})$, $H_{2}(\mathcal{G})$, $H^{a}_{3}(\mathcal{G})$, $H^{n}_{3}(\mathcal{G})$, $U_{3}(\mathcal{G})$ are disjoint.

\begin{lemma}\label{lem:1.3}
In a MAG $\mathcal{G}$, for any single vertex $a$, $\tail(a)=\pa_{\mathcal{G}}(a)$, and $\{v,w\}$ is a head if and only if $v \leftrightarrow w$.
\end{lemma}

\begin{proof}
If $ a \subset \dis_{\an(a)}(a)$ then there is a vertex $b$ such that $b \leftrightarrow a$ and $b \in \an_{\mathcal{G}}(a)$, which contradicts ancestrality. Hence $\tail(a)=\pa_{\mathcal{G}}(a)$.
\par
If $v \leftrightarrow w$ then $v,w$ have no ancestral relation so by definition, it is a head. Suppose $\{v,w\}$ is a head, so $\{v,w\} \in \mathcal{S}(\mathcal{G})$ then they must be adjacent by Proposition 3.4 and the adjacency can not be undirected or directed, thus $v \leftrightarrow w$.
\end{proof}

Thus $H_{1}(\mathcal{G})$ and $H_{2}(\mathcal{G})$ are precisely the sets in $\mathcal{S}_{3}(\mathcal{G})$ that arise from heads of size one and two, respectively.

\begin{lemma}\label{Sequality}
For a MAG $\mathcal{G}$, we have
\begin{align*}
    \mathcal{S}_{3}(\mathcal{G}) &= \mathcal{S}_{2}(\mathcal{G}) \cup H_{1}(\mathcal{G}) \cup H_{2}(\mathcal{G}) \cup H_{3}(\mathcal{G}) \cup U_{3}(\mathcal{G})\\
    \hat{\mathcal{S}}_{3}(\mathcal{G}) &= \mathcal{S}_{2}(\mathcal{G}) \cup H_{1}(\mathcal{G}) \cup H_{2}(\mathcal{G}) \cup H^{a}_{3}(\mathcal{G}) \cup U_{3}(\mathcal{G}).
\end{align*}
\end{lemma}

\begin{proof}
Consider the first equality, for $S = \{v,w\} \in \mathcal{S}_{3}(\mathcal{G})$, by Proposition 3.4, $v,w$ are adjacent in $\mathcal{G}$ so $S \in \mathcal{S}_{2}$; For $S \in \mathcal{S}_{3}(\mathcal{G})$ and $\abs{S}=3$, it is a clique in $\mathcal{G}^{u}$ or it origins from heads of size either 1 or 2 or 3. Thus by Lemma 4.1 and Lemma 4.1, $S \in H_{1}(\mathcal{G}) \cup H_{2}(\mathcal{G}) \cup H_{3}(\mathcal{G}) \cup U_{3}(\mathcal{G})$; For $S$ in the right hand side, it is in $\mathcal{S}_{3}(\mathcal{G})$ by definition.
\par
For the second equality, by definition $\hat{\mathcal{S}}_{3}(\mathcal{G})$ excludes all $S \in \mathcal{S}_{3}(\mathcal{G})$ that have no adjacencies, but note that all $S \in \mathcal{S}_{2}(\mathcal{G}) \cup H_{1}(\mathcal{G}) \cup H_{2}(\mathcal{G}) \cup U_{3}(\mathcal{G})$ have some adjacencies. And by definition $H^{a}_{3}(\mathcal{G})$ extract all heads of size 3 with some adjacencies.
\end{proof}

\begin{lemma}\label{A'}
For a MAG $\mathcal{G}$, $A'_{1}(\mathcal{G}) \cup U_{3}(\mathcal{G}) = \hat{\mathcal{S}}_{3}(\mathcal{G})$.
\end{lemma}

\begin{proof}
%$S \in A'_{1}(\mathcal{G})$ obtained at line 4, 6, 8, 11, 15 and 16, correspond to sets in 
$S \in A'_{1}(\mathcal{G})$ obtained at line 5, 7, 9, 12, 17 and 18, correspond to sets in 
$\mathcal{S}_{2}(\mathcal{G})$, $H_{1}(\mathcal{G})$, $\mathcal{S}_{2}(\mathcal{G})$, $H_{2}(\mathcal{G})$, $H^{a}_{3}(\mathcal{G})$ and $\mathcal{S}_{2}(\mathcal{G})$, respectively. So by Lemma \ref{Sequality}, $A'_{1}(\mathcal{G}) \cup U_{3}(\mathcal{G}) \subseteq \hat{\mathcal{S}}_{3}(\mathcal{G})$ Conversely, all sets in $\hat{\mathcal{S}}_{3}(\mathcal{G}) \setminus U_{3}(\mathcal{G})$ can be obtained at corresponding lines.
\end{proof}

\begin{proposition41}
For a MAG $\mathcal{G}$, $A_{1}(\mathcal{G}) = \Tilde{\mathcal{S}}_{3}(\mathcal{G})$.
\end{proposition41}
 
 \begin{proof}
 Compared to $\hat{\mathcal{S}}_{3}(\mathcal{G})$, $\Tilde{\mathcal{S}}_{3}(\mathcal{G})$ excludes all sets of size 3 that have 3 adjacencies. If the set is clique in $\mathcal{G}^{u}$ except for edges, it is not added in Algorithm 1. Otherwise note that when sets of size 3 are obtained, 
%%lines 5, 10 and 12 check their adjacencies.
lines 6, 11 and 14 check their adjacencies.
\end{proof}
\par
Notice that Algorithm 1 naturally identifies $\hat{\mathcal{S}}_{3}(\mathcal{G}) \setminus U_{3}(\mathcal{G})$, but to obtain the full $\hat{\mathcal{S}}_{3}(\mathcal{G})$ one also needs to identify all triangles in the undirected component; $\Tilde{\mathcal{S}}_{3}(\mathcal{G})$ excludes this set. 
\begin{proposition42}
Let $A_i$ be the number of ancestors of the vertex $i$.  Then
\begin{align*}
\mathbb{E} A_i = \left( 1 + \frac{r}{n} \right)^{i-1}.
\end{align*}
In particular,
\begin{align*}
\mathbb{E} A_n = \left( 1 + \frac{r}{n} \right)^{n-1} \longrightarrow e^r.
\end{align*}
\end{proposition42}
\begin{proof}
We proceed by induction.  The result is trivially true for $A_2 = 1 + \frac{r}{n}$.
Suppose the result holds for $A_j$.  Then
\begin{align*}
\mathbb{E} A_{j+1} &= 1 + \sum_{i=1}^{j} \mathbb{E} \mathbbm{1}_{\{i \rightarrow j+1\}} A_i\\
 &= 1 + \frac{r}{n} \sum_{i=1}^{j} \left(1 + \frac{r}{n} \right)^{i-1},
\end{align*}
using independence of the edge and $A_i$ and the induction hypothesis.  Hence
\begin{align*}
\mathbb{E} A_{j+1} &= 1 + \sum_{i=1}^{j} \sum_{k=0}^{i-1} {i-1 \choose k} \left(\frac{r}{n} \right)^{k+1}\\
&= 1 + \sum_{k=0}^{j-1} \left(\frac{r}{n} \right)^{k+1} \sum_{i=k+1}^{j-1} {i-1 \choose k} \\
&= 1 + \sum_{k=0}^{j-1} \left(\frac{r}{n} \right)^{k+1} {j \choose k+1}\\
&= 1 + \sum_{k=1}^{j} \left(\frac{r}{n} \right)^{k} {j \choose k}.
\end{align*}
by a standard result about binomial coefficients.  
This gives the result.
\end{proof}

\end{appendices}

% \label{app:theorem}

% % Note: in this sample, the section number is hard-coded in. Following
% % proper LaTeX conventions, it should properly be coded as a reference:

% %In this appendix we prove the following theorem from
% %Section~\ref{sec:textree-generalization}:

% In this appendix we prove abcdefg:

\newpage

\bibliographystyle{abbrvnat}
\bibliography{refs}

\end{document}